\documentclass[12pt]{article}
\usepackage[top=2cm,bottom=2cm,left=2cm,right=2cm]{geometry}

\usepackage{amsthm,amsmath,amssymb}

\usepackage{graphicx}

\usepackage[colorlinks=true,citecolor=black,linkcolor=black,urlcolor=blue]{hyperref}

\newcommand{\doi}[1]{\href{http://dx.doi.org/#1}{\texttt{doi:#1}}}
\newcommand{\arxiv}[1]{\href{http://arxiv.org/abs/#1}{\texttt{arXiv:#1}}}

\theoremstyle{plain}
\newtheorem{theorem}{Theorem}
\newtheorem{lemma}[theorem]{Lemma}
\newtheorem{corollary}[theorem]{Corollary}

\newtheorem{fact}[theorem]{Fact}

\theoremstyle{definition}

\newtheorem{conjecture}[theorem]{Conjecture}

\theoremstyle{remark}

\title{\bf Maximal antichains of minimum size}

\author{Thomas Kalinowski \and Uwe Leck \and Ian T. Roberts}

\begin{document}

\maketitle


\begin{abstract}
Let $n\geqslant 4$ be a natural number, and let $K$ be a set $K\subseteq [n]:=\{1,2,\dots,n\}$. We study the problem to find the smallest possible size of a maximal family $\mathcal{A}$ of subsets of $[n]$ such that $\mathcal{A}$ contains only sets whose size is in $K$, and $A\not\subseteq B$ for all $\{A,B\}\subseteq\mathcal{A}$, i.e. $\mathcal{A}$ is an antichain. We present a general construction of such antichains for sets $K$ containing 2, but not 1. If $3\in K$ our construction asymptotically yields the smallest possible size of such a family, up to an $o(n^2)$ error. We conjecture our construction to be asymptotically optimal also for $3\not\in K$, and we prove a weaker bound for the case $K=\{2,4\}$. Our asymptotic results are straightforward applications of the graph removal lemma to an equivalent reformulation of the problem in extremal graph theory which is interesting in its own right. 

\end{abstract}

\section{Introduction}\label{sec:intro}
Let $n\geqslant 2$ be an integer and $[n]:=\{1,2,\dots,n\}$. By $2^{[n]}$ we denote the family of all subsets of $[n]$ and by $\binom{[n]}k$ the family of all $k$-subsets of $[n]$. A family $\mathcal{A}\subseteq 2^{[n]}$ is an \emph{antichain} if $A\not\subseteq B$ for all distinct $A,B\in\mathcal{A}$. An antichain $\mathcal A$ is called \emph{flat} if it is contained in two consecutive levels of $2^{[n]}$, i.e. if $\mathcal A\subseteq\tbinom{[n]}{k}\cup\tbinom{[n]}{k+1}$ for some $k$. More generally, for a subset $K\subseteq[n]$, we call $\mathcal A$ a \emph{$K$-antichain} if it contains only $k$-sets with $k\in K$, i.e. 
\[\mathcal A=\bigcup\limits_{k\in K}\mathcal A_k,\] 
where $\mathcal A_k\subseteq\tbinom{[n]}{k}$. The \emph{dual} $\mathcal A^*$ of a family of sets $\mathcal A = \{A_1, \ldots, A_m\}$ is the collection $\mathcal A^* = \{ X_1, \ldots, X_n\}$ of subsets of $[m]$ given by $X_i = \{ j\in[m]\ :\ i \in A_j\}$ for $i\in[n]$. It is well known that a family of sets $\mathcal A$ is an antichain if and only if its dual $\mathcal A^*$ is a \emph{completely separating system} (CSS): a CSS is a collection $\mathcal C$ of blocks of $[n]$  such that for each $a,b \in [n]$ there are blocks  $A,B \in\mathcal C$ with $a \in A\setminus B$ and $ b \in B\setminus A$. The dual of a flat antichain on $\binom{[n]}{k} \cup \binom{[n]}{k-1}$  is a \emph{fair CSS}: a CSS in which each point occurs exactly $k$ or $k-1$ times. The consideration of minimum size CSSs led to a conjecture which subsequently became the \emph{Flat Antichain Theorem}, which follows from results of Lieby \cite{lieby1999extremal} (see also \cite{lieby2004antichains}) and Kisv{\"o}lcsey \cite{kisvolcsey2006flattening}. This theorem greatly reduced the computational search space for CSSs. It says that for every antichain $\mathcal A$ there is an equivalent flat antichain $\mathcal A'$, where the equivalence relation is defined by: $\mathcal A'$ is equivalent to $\mathcal A$  if and only if $\lvert\mathcal A'\rvert=\lvert\mathcal A\rvert$ and $\sum_{A\in\mathcal A'}\lvert A\rvert=\sum_{A\in\mathcal A}\lvert A\rvert$. Griggs et al. \cite{GriggsHartmannLeckRoberts2012} showed that the flat antichains minimize the \emph{BLYM-values} $\sum_{A\in\mathcal A}\tbinom{n}{\lvert A\rvert}^{-1}$ within their equivalence classes, and more generally, they minimize (maximize) $\sum_{A\in\mathcal A}w(\lvert A\rvert)$ for every convex (concave) function $w$. CSSs have applications in various areas including Coding Theory, Search Theory, and Topology. See~\cite{du2000combinatorial} and~\cite{gruttmuller2012antichains} for applications and further references for CSSs.

A $K$-antichain $\mathcal A$ is called a \emph{maximal $K$-antichain} if there is no $K$-antichain $\mathcal A'$ with $\mathcal A\subsetneq \mathcal A'$. For $K'\subseteq K$, any $K'$-antichain is also a $K$-antichain, and if it is a maximal $K$-antichain, then it is also a maximal $K'$-antichain. The converse is not true as can be seen by the family
\[\mathcal A=\{1245,2367,1389,16,17,28,29,34,35,46,47,48,49,56,57,58,59,68,69,78,79\}\]
for $n=9$, which is a maximal $\{2,4\}$-antichain but not a maximal $\{2,3,4\}$-antichain, since $\mathcal A\cup\{123\}$ is an antichain properly containing $\mathcal A$. This example is illustrated in Figure~\ref{fig:maximality}, showing a graph on 9 vertices whose edge set is the complement of $\mathcal A_2$ in $\tbinom{[9]}{2}$, while the set of 4-cliques is precisely $\mathcal A_4$.
\begin{figure}[htb]
  \centering
  \includegraphics[width=.4\linewidth]{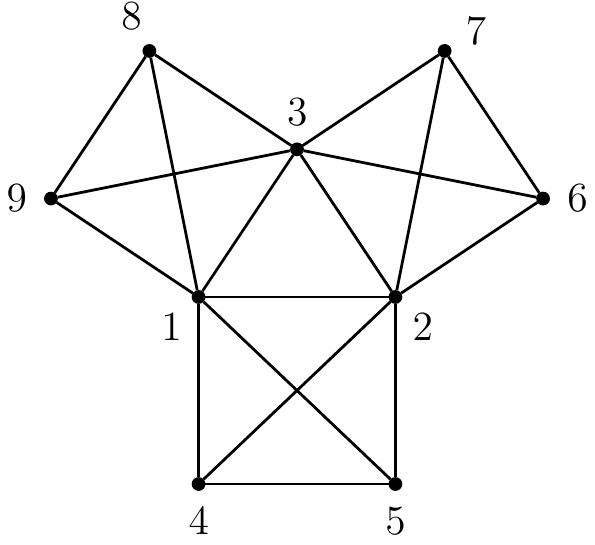}
  \caption{A graph representing a maximal $\{2,4\}$-antichain which is not maximal as a $\{2,3,4\}$-antichain, and thus also not strongly maximal.}
  \label{fig:maximality}
\end{figure}
We also call a $K$-antichain $\mathcal A$ \emph{strongly maximal} if it is maximal as an antichain, not just as $K$-antichain, i.e. if there is no antichain $\mathcal A'$ with $\mathcal A\subsetneq \mathcal A'$. Equivalently, $\mathcal A$ is a strongly maximal $K$-antichain if it is also a maximal $\{k,k+1,\ldots,l\}$-antichain, where $k$ and $l$ are the smallest and the largest element of $K$, respectively. In this paper, we always assume $1\not\in K$ and $2\in K$, and we study the problem of determining the smallest possible cardinality of a maximal $K$-antichain. 

Our approach to the minimum size of a maximal antichain is based on the graph interpretation which was used in~\cite{GruettmuellerHartmannKalinowskiLeckRoberts2009} to completely resolve the problem for the case $K=\{2,3\}$. To prove our asymptotic results we adapt arguments from~\cite{gerbner2011saturating} where lower bounds for the size of maximal $\{l,l+1\}$-antichains were proved. The main tool in this part is the graph removal lemma. 
 \begin{fact}[Graph removal lemma.] Let $H$ be a graph on $h$ vertices and let $\varepsilon>0$. Then there exists $\delta>0$ such that every graph $G$ on $n$ vertices containing at most $\delta n^h$ copies of $H$ can be made $H$-free by removing $\varepsilon n^2$ edges.   
 \end{fact} 
This was proved, using Szemer\'edis regularity lemma, first for $H$ being a triangle~\cite{RuszaSzemeredi1976} and later in general~\cite{ErdoesFranklRoedl1986}. See~\cite{Fox2010} for a recent proof avoiding the regularity lemma, and thus improving the bounds on $\delta$ significantly. 

The paper is organized as follows. In Section~\ref{sec:graph} the problem is reformulated as a question in extremal graph theory. In Section~\ref{sec:construction} we describe a construction for graphs, which we believe to correspond to optimal antichains. Some evidence for this conjecture is provided by exhaustive search results for $K=\{2,4\}$ and small $n$ which are presented in Section~\ref{sec:small_n}. Section~\ref{sec:asymptotics} contains our asymptotic results: if $3\in K$  our construction yields the correct leading term, and for $K=\{2,4\}$ we prove a first nontrivial bound. Finally, in Section~\ref{sec:problems} we suggest some open problems for further investigations. 

\section{Graph theoretical reformulation}\label{sec:graph}
As proposed in~\cite{GruettmuellerHartmannKalinowskiLeckRoberts2009}, we associate a graph $G(\mathcal A)=(V,E)$ with an antichain $\mathcal A$. The vertex set is $V=[n]$ and the edge set
\[E=\binom{[n]}{2}\setminus\mathcal A_2.\]
We start by making some simple observations.
\begin{lemma}\label{lem:simple_observations}
Let $\mathcal A$ be a $K$-antichain on $n$ points and let $G=(V,E)$ be the associated graph.
\begin{enumerate}
\item Every $B\in\mathcal A\setminus\mathcal A_2$ induces a clique in $G$.
\item If $\mathcal A$ is a maximal $K$-antichain and $A$ is the vertex set of a $k$-clique in $G$ for some $k\in K$ then there is an $A'\in\mathcal A$ with $A\subseteq A'$ or $A\supseteq A'$. In particular, every edge of $G$ is contained in a $k$-clique for some $k\in K\setminus\{2\}$.
\end{enumerate}
\end{lemma}
\begin{proof}
The first statement is just the antichain property applied to $A\in\mathcal A_2=\tbinom{[n]}{2}\setminus E$ and $B\in\mathcal A\setminus\mathcal A_2$. The second statement follows from maximality, because if no such $A'$ exists then we can add $A$ to $\mathcal A$ to obtain a strictly larger antichain.
\end{proof}
The elements of the antichain are the non-edges of the graph and some cliques whose sizes are in $K\setminus\{2\}$. We say that a graph is \emph{$K$-saturated} if every edge is contained in a $k$-clique for some $k\in K\setminus\{2\}$. Being $K$-saturated is a necessary and sufficient condition for a graph to be $G(\mathcal A)$ for some maximal $K$-antichain $\mathcal A$. In general, it is not possible to reconstruct the maximal $K$-antichain uniquely from the graph: if $\{3,4\}\subseteq K$ then a 4-clique can (among others) come from a single 4-set or from four 3-sets in $\mathcal A$. 
For a $K$-saturated graph $G=(V,E)$ with $V=[n]$ let $\mathfrak A(G)$ denote the set of all $K$-maximal antichains $\mathcal A$ with $G(\mathcal A)=G$. In other words $\mathfrak A(G)$ is the collection of all families $\mathcal A=\bigcup_{k\in K}\mathcal A_k\subseteq 2^{[n]}$ such that
\begin{enumerate}
\item $\mathcal A_2=\binom{[n]}{2}\setminus E$,
\item for all $k\in K\setminus\{2\}$ every $A\in\mathcal A_k$ induces a $k$-clique in $G$, and
\item for all $k\in K\setminus\{2\}$ and every $k$-clique $A$ in $G$, there is a $k'\in K$ and an $A'\in\mathcal A_{k'}$ with $A\subseteq A'$ or $A'\subseteq A$.
\end{enumerate}

A greedy choice among all the maximal $K$-antichains corresponding to a given $K$-saturated graph $G$ is based on the intuition to take the sets in $\mathcal A$ as large as possible for the given graph $G$. More precisely, for $k\in K$, let $\mathcal C_k$ denote the set of $k$-cliques in $G$ that are not contained in a $k'$-clique for any $k'\in K$ with $k'>k$. Note that for $K$-saturated graphs $\mathcal C_2=\varnothing$. We put $\mathcal M=\bigcup\limits_{k\in K}\mathcal C_k$ and define the maximal $K$-antichain $\mathcal A(G)\in\mathfrak A(G)$ by 
\begin{equation}
  \label{eq:canonical_ac}
\mathcal A(G)=\left(\binom{[n]}{2}\setminus E\right)\cup\mathcal M.  
\end{equation}
It is not necessarily the smallest maximal $K$-antichain corresponding to the graph $G$: for $K=\{2,3,4\}$ and $G$ the complete graph on $n$ vertices we obtain $\mathcal A(G)=\tbinom{[n]}{4}$, while $\tbinom{[n]}{3}$ is a smaller maximal $K$-antichain corresponding to the same graph. But under the following sparseness condition on $G$ the antichain $\mathcal A(G)$ has minimum size. Let $k^*$ be the minimal element of $K\setminus\{2\}$. We say that a $K$-saturated graph is \emph{$K$-sparse} if for every $A\in\mathcal M$, there exists a $k^*$-clique $B\subseteq A$ such that $A$ is the unique element of $\mathcal M$ containing $B$.
The next lemma asserts that in this situation the antichain $\mathcal A(G)$ has minimum size among all $K$-antichains with corresponding graph $G$. 
\begin{lemma}\label{lem:canonical_AC}
Let $G=(V,E)$ be a $K$-sparse, $K$-saturated graph, and let $\mathcal A=\mathcal A(G)$. Then $\lvert\mathcal A\rvert\leqslant\lvert \mathcal A'\rvert$ for all maximal $K$-antichains $\mathcal A'\in\mathfrak A(G)$.  
\end{lemma}
\begin{proof}
Let $\mathcal A'$ be any maximal $K$-antichain with $G(\mathcal A')=G$. Clearly $\mathcal A'_2=\mathcal A_2=\tbinom{[n]}{2}\setminus E$. Let $\mathcal C^*$ be the set of $k^*$-cliques in $G$ that are contained in exactly one element of $\mathcal M$. Since $G$ is $K$-sparse there is a mapping $f:\mathcal M\to\mathcal C^*$ with $f(A)\subseteq A$ for every $A\in\mathcal M$. By the definition of $\mathcal C^*$, the mapping $f$ is injective. By the second part of Lemma \ref{lem:simple_observations} every element of the image of $f$ is contained in some member of $\mathcal A'$. If the sets $f(A_1)$ and $f(A_2)$ are contained in the same $A\in\mathcal A'$, then $A\subseteq A_i$ for $i\in\{1,2\}$ because $f(A_i)$ is contained in a unique element of $\mathcal M$. This, in turn, implies $A_1=A_2$ for the same reason. Hence $\lvert\mathcal M\rvert$ is a lower bound for $\lvert \mathcal A'\setminus\mathcal A'_2\rvert$, and this proves the statement.      
\end{proof}
The following lemma provides a characterization of the $K$-saturated graphs $G$ such that $\mathcal A(G)$ is a strongly maximal $K$-antichain.
\begin{lemma}\label{lem:strong_maximality}
Let $G$ be a $K$-saturated graph for $K=\{2=k_1<k_2<\cdots<k_r\}$. Then $\mathcal A=\mathcal A(G)$ is strongly maximal if and only if the following holds. For every $i\in\{1,2,\ldots,r-1\}$, every $q\in\{k_i+1,\ldots,k_{i+1}-1\}$ and every $q$-clique $B$ in $G$, we have that either
\begin{enumerate}
\item\label{cond_1} $B$ contains a $k_i$-clique that is not contained in any $k_{i+1}$-clique, or
\item\label{cond_2} $B$ is contained in a $k_{i+1}$-clique.
\end{enumerate}
\end{lemma}
\begin{proof}
Suppose $G=G(\mathcal A)$ satisfies the condition and $\mathcal A$ is not strongly maximal. Then there is a set $B\not\in\mathcal A$ that can be added such that $\mathcal A\cup\{B\}$ is still an antichain. The set $B$ must induce a clique in $G$ and by Lemma \ref{lem:simple_observations} this implies $\lvert B\rvert\not\in K$ and $\lvert B\rvert< k_r$. Hence $k_i<\lvert B\rvert <k_{i+1}$ for some $i$. But then either of the conditions \ref{cond_1} (with Lemma \ref{lem:simple_observations}) and \ref{cond_2} (using the definition of $\mathcal A(G)$) implies that $B$ cannot be added without violating the antichain condition. For the converse, suppose $G$ is a $K$-saturated graph containing a $q$-clique $B$, $k_i<q<k_{i+1}$, such that $B$ is not contained in any $k_{i+1}$-clique but every $k_{i}$-subset of $B$ is contained in a $k_{i+1}$-clique. By construction of $\mathcal A=\mathcal A(G)$, for every $k_i$-subset $A\subseteq B$, the antichain $\mathcal A$ contains a $k_j$-set $A'\supset A$ for some $j>i$. But this implies that $\mathcal A\cup\{B\}$ is still an antichain, hence $\mathcal A$ is not strongly maximal. 
\end{proof}

To summarize, we have proved the following.
\begin{enumerate}
\item For every maximal $K$-antichain $\mathcal A$ there is a unique $K$-saturated graph $G(\mathcal A)$.
\item Minimizing the size of a maximal $K$-antichain is equivalent to 
\begin{equation}\label{eq:reformulation}
\max\limits_{G}\left(\lvert E\rvert-\min_{\mathcal A\in\mathfrak A(G)}\sum_{k\in K\setminus\{2\}}\lvert\mathcal A_k\rvert\right),
\end{equation}
where the maximum is over the set of all $K$-saturated graphs $G$ and the minimum is over the set $\mathfrak A(G)=\{\mathcal A\ :\ G(\mathcal A)=G\}$.
\item If the graph $G$ is $K$-sparse then the antichain $\mathcal A=\mathcal A(G)$ defined by (\ref{eq:canonical_ac}) achieves the minimum in (\ref{eq:reformulation}).
\end{enumerate}

\section{Constructing $K$-saturated graphs}\label{sec:construction}
\subsection{A general construction}\label{subsec:general_construction}
In this section, we describe a family of $K$-saturated graphs with many edges and few cliques (corresponding to small maximal $K$-antichains). Our construction is inspired by the optimal graphs for $\{2,3\}$-antichains in \cite{GruettmuellerHartmannKalinowskiLeckRoberts2009}: We start with a complete bipartite graph with almost equal parts and add a few edges to make the graph $K$-saturated without producing too many cliques with sizes in $K$. Let $l$ be the maximal element of $K$. The vertex set is partitioned into three disjoint sets $A$, $B$, $C$ with
\begin{align*}
\lvert A\rvert &= \left\lfloor\frac{\lfloor n/2\rfloor}{\lfloor l/2\rfloor}\right\rfloor\cdot\lfloor l/2\rfloor, &
\lvert B\rvert &= \left\lfloor\frac{\lceil n/2\rceil}{\lceil l/2\rceil}\right\rfloor\cdot\lceil l/2\rceil, & 
\lvert C\rvert &= n-\lvert A\rvert-\lvert B\rvert<l,
\end{align*}
and the edge set is determined as follows. 
\begin{itemize}
\item The induced subgraph on $A$ is the disjoint union of $\lfloor l/2\rfloor$-cliques.
\item The induced subgraph on $B$ is the disjoint union of $\lceil l/2\rceil$-cliques.
\item Every pair $ab$ with $a\in A$ and $b\in B$ is an edge.
\item The vertices in $C$ are isolated.
\end{itemize}
Clearly every clique of at least two vertices is contained in an $l$-clique, so $\mathcal M$ is the set of all $l$-cliques. Any $l$-clique contains vertices from $A$ and vertices from $B$, and every $k^*$-clique that has nonempty intersection with $A$ and with $B$ is contained in a unique $l$-clique. Thus the graph $G$ is $K$-sparse and using Lemma \ref{lem:canonical_AC} the antichain $\mathcal A(G)$, defined by (\ref{eq:canonical_ac}), has minimum size among the antichains in $\mathfrak A(G)$ and contains only $2$- and $l$-sets. We get $\lvert E\rvert=(1/4+O(1/n))n^2$ and
\[\lvert \mathcal C_l\rvert=\left\lfloor\frac{\lfloor n/2\rfloor}{\lfloor l/2\rfloor}\right\rfloor\cdot\left\lfloor\frac{\lceil n/2\rceil}{\lceil l/2\rceil}\right\rfloor=\left(\frac{1}{4\lfloor l/2\rfloor\lceil l/2\rceil}+O(1/n)\right)n^2.\]
Thus
\[\lvert E\rvert-\lvert\mathcal C_l\rvert=\left(\frac{\lfloor l/2\rfloor\lceil l/2\rceil-1}{4\lfloor l/2\rfloor\lceil l/2\rceil}+o(1)\right)n^2=
\begin{cases}
\left(\frac{l^2-4}{4l^2}+O(1/n)\right)n^2  & \text{for even }l,\\
\left(\frac{l^2-5}{4l^2-4}+O(1/n)\right)n^2  & \text{for odd }l.
\end{cases}\]
The corresponding maximal $K$-antichain has size  
\[\left(\frac{\lfloor l/2\rfloor\lceil l/2\rceil+1}{4\lfloor l/2\rfloor\lceil l/2\rceil}+O(1/n)\right)n^2.\]
Note that in this construction every $q$-clique for $2\leqslant q\leqslant l$ is contained in an $l$-clique, thus by Lemma \ref{lem:strong_maximality} the antichains $\mathcal A(G)$ are strongly $K$-maximal. 

\subsection{The case $l=4$}\label{subsec:l_4}
In the above construction we wasted the vertices in the set $C$, and the solution can be slightly improved by connecting $C$ to the rest of the graph. The best way to do this depends on the remainder of $n$ modulo $l$. Here we describe what to do for $l=4$ and $n=4m+r$ ($0\leqslant r\leqslant 3$).
\begin{description}
\item[$n=4m$.] In this case $A=[2m]$, $B=[2m+1,4m]$ and $C=\varnothing$ and the edge set is
\[E=\left\{\{i,j\}\ :\ i\in[2m],\,j\in[2m+1,4m]\right\}\cup\{\{2i-1,2i\}\ :\ i\in[2m]\},\]
and we have
\[\lvert E\rvert-\lvert\mathcal C_4\rvert=\frac{n^2}{4}+\frac{n}{2}-\frac{n^2}{16}=\frac{3n^2+8n}{16}.\]
\item[$n=4m+1$.] In this case $A=[2m]$, $B=[2m+1,4m]$ and $C=\{4m+1\}$ and the edge set is
\begin{multline*}
E=\left\{\{i,j\}\ :\ i\in[2m],\,j\in[2m+1,4m+1]\right\}\cup\left\{\{2i-1,2i\}\ :\ i\in[2m]\right\}\\ \cup\{\{4m,4m+1\}\},
\end{multline*}
and we have
\[\lvert E\rvert-\lvert\mathcal C_4\rvert=\frac{n^2-1}{4}+\frac{n+1}{2}-\frac{n-1}{4}\cdot\frac{n+3}{4}=\frac{3n^2+6n+7}{16}.\]
\item[$n=4m+2$.] In this case $A=[2m]$, $B=[2m+1,4m+2]$ and $C=\varnothing$ and the edge set is
\[E=\left\{\{i,j\}\ :\ i\in[2m],\,j\in[2m+1,4m+2]\right\}\cup\{\{2i-1,2i\}\ :\ i\in[2m+1]\},\]
and we have
\[\lvert E\rvert-\lvert\mathcal C_4\rvert=\frac{n^2-4}{4}+\frac{n}{2}-\frac{n-2}{4}\cdot\frac{n+2}{4}=\frac{3n^2+8n-12}{16}.\]
\item[$n=4m+3$.] In this case $A=[2m]$, $B=[2m+1,4m+2]$ and $C=\{4m+3\}$ and the edge set is
\begin{multline*}
E=\left\{\{i,j\}\ :\ i\in[2m]\cup\{4m+3\},\,j\in[2m+1,4m+2]\right\}\cup\{\{2i-1,2i\}\ :\ i\in[2m+1]\}\\ \cup\left\{\{2m,4m+3\}\right\},
\end{multline*}
and we have
\[\lvert E\rvert-\lvert\mathcal C_4\rvert=\frac{n^2-1}{4}+\frac{n+1}{2}-\frac{n+1}{4}\cdot\frac{n+1}{4}=\frac{3n^2+6n+5}{16}.\]
\end{description}
The construction is illustrated in Figure \ref{fig:examples}.
\begin{figure}[htb]
\begin{minipage}{.48\textwidth}
  \centering
  \includegraphics[width=\textwidth]{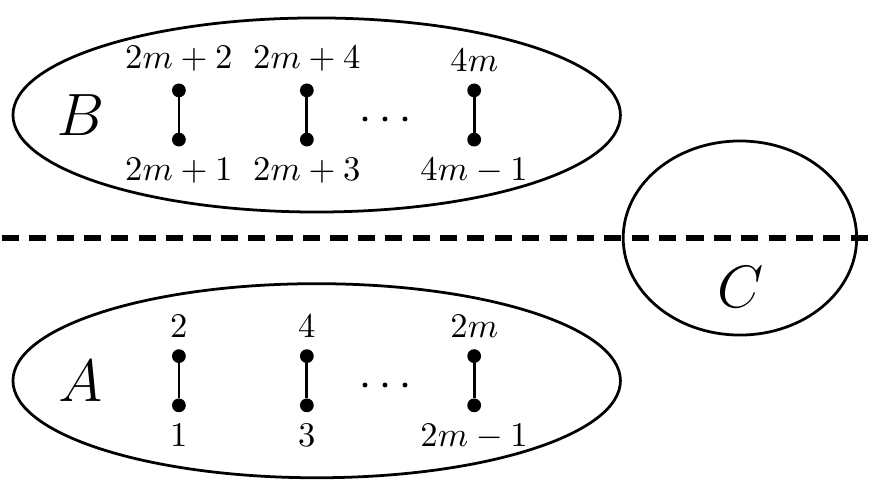}
\end{minipage}\hfill
\begin{minipage}{.48\textwidth}
  \centering
  \includegraphics[width=\textwidth]{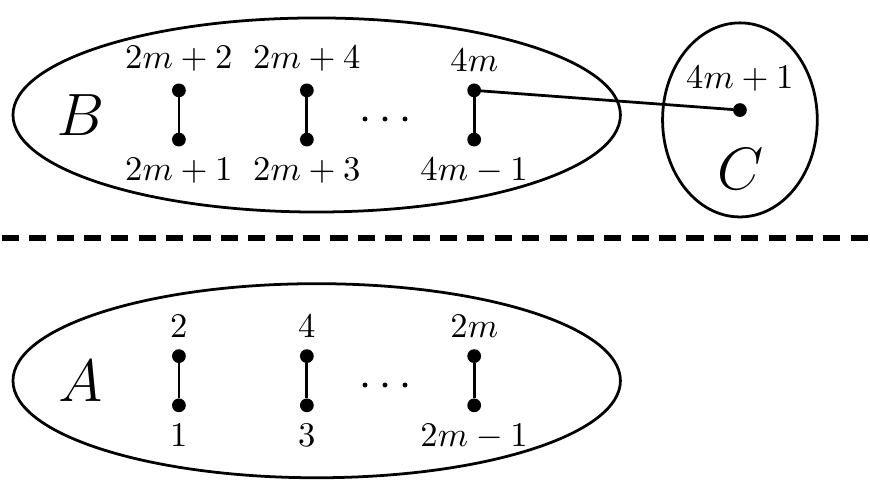}
\end{minipage}

\bigskip\bigskip

\begin{minipage}{.48\textwidth}
  \centering
  \includegraphics[width=\textwidth]{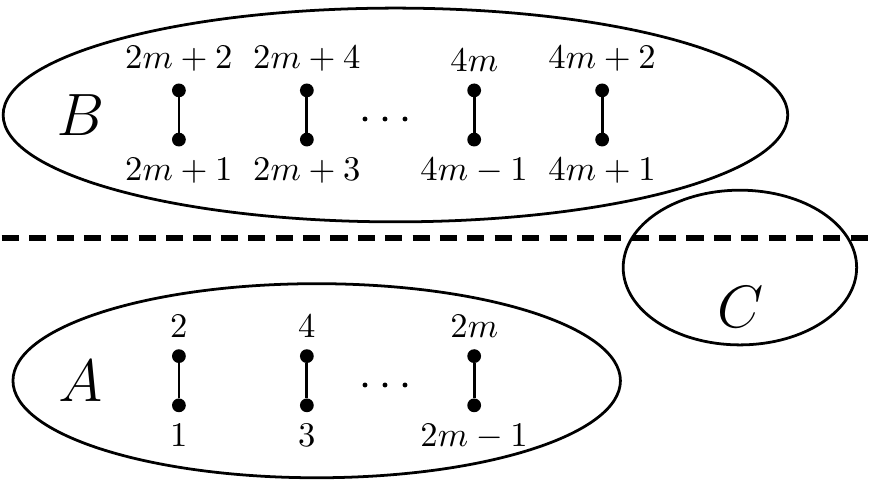}
\end{minipage}\hfill
\begin{minipage}{.48\textwidth}
  \centering
  \includegraphics[width=\textwidth]{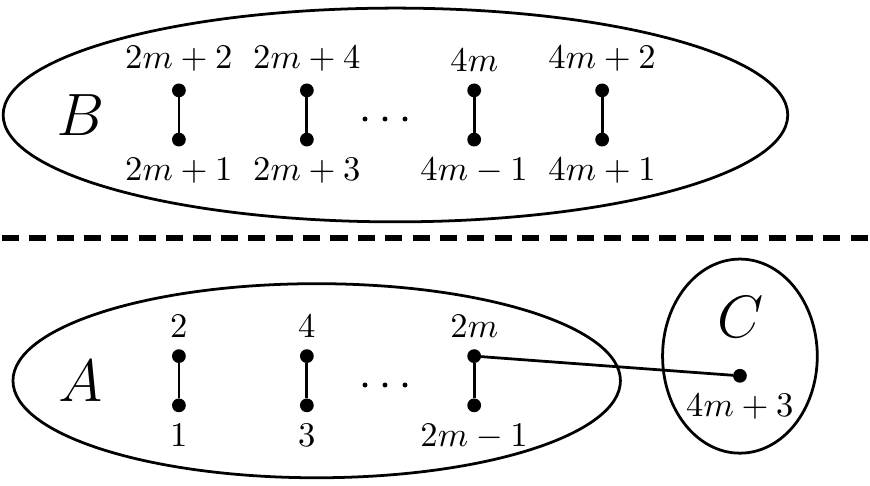}
\end{minipage}
\caption{$4$-saturated graphs for different values of $n\pmod 4$. All pairs of vertices on different sides of the dashed line are joined by an edge, but these edges are omitted in the pictures.}\label{fig:examples}
\end{figure}
We conjecture that for large $n$ these graphs are best possible.
\begin{conjecture}\label{conj:graph_conjecture}
If $n$ is large, in any $\{2,4\}$-saturated graph $G$ on $n$ vertices, we have
\[\lvert E\rvert-\lvert\mathcal C_4\rvert\leqslant f(n) :=
\begin{cases}
\left\lfloor(3n^2+8n)/16\right\rfloor & \text{if $n$ is even,}\\
\left\lceil(3n^2+6n)/16\right\rceil & \text{if $n$ is odd.}\\
\end{cases}\]   
\end{conjecture}
In terms of antichains the equivalent statement is that for any maximal $\{2,4\}$-antichain $\mathcal A$ on $n$ points, we have 
\[\lvert\mathcal A\rvert\geqslant\begin{cases}
\left\lceil (5n^2-16n)/16\right\rceil & \text{if $n$ is even,}\\
\left\lfloor(5n^2-14n)/16\right\rfloor & \text{if $n$ is odd.}\\
\end{cases}\]  
\section{Small values of $n$}\label{sec:small_n}
In this section we present some computational results for $K\subseteq\{2,3,4\}$ and $n\leqslant 16$. 
These results, summarized in Table \ref{tab:small_cases}, were obtained by solving binary programs for maximizing $\lvert E\rvert-\lvert\mathcal A_4\rvert$ and $\lvert E\rvert-\lvert\mathcal A_3\rvert-\lvert\mathcal A_4\rvert$, respectively, over all pairs $(G,\mathcal A)$ of a $K$-saturated graph $G$ on $n$ vertices and an antichain $\mathcal A\in\mathfrak A(G)$. The binary programs were solved with the commercial MIP solver CPLEX.   
\begin{table}[htb]
  \centering
  \begin{tabular}[t]{r|rr|rr|r} \hline \hline
 & \multicolumn{2}{c|}{$K=\{2,4\}$} & \multicolumn{2}{c|}{$K=\{2,3,4\}$} &  \\ 
$n$ & $\lvert\mathcal A\rvert$ & profile & $\lvert\mathcal A\rvert$ & profile & $\binom{n}{2}-f(n)$ \\ \hline
4 & 1 & $(0,1)$ & 1 & $(0,0,1)$ & 1 \\
5 & 3 & $(1,2)$ & 3 & $(1,0,2)$ & 3 \\
6 & 6 & $(2,4)$ & 6 & $(2,0,4)$ & 6 \\
7 & 9 & $(5,4)$ & 9 & $(5,0,4)$ & 9 \\
8 & 12 & $(8,4)$ & 12 & $(8,0,4)$ & 12 \\
9 & 17 & $(11,6)$ & 17 & $(11,0,6)$ & 17 \\
10 & 20 & $(15,5)$ & 22 & $(16,0,6)$ & 22 \\
11 & 26 & $(19,7)$ & 28 & $(19,0,9)$ & 28 \\
12 & 33 & $(24,9)$ & 33 & $(24,0,9)$ & 33 \\
13 & 40 & $(29,11)$ & 41 & $(29,0,12)$ & 41 \\
14 & 47 & $(35,12)$ & 48 & $(36,0,12)$ & 48 \\
15 & 56 & $(41,15)$ & 57 & $(41,0,16)$ & 57 \\
16 & 64 & $(48,16)$ & 64 & $(48,0,16)$ & 64 \\ \hline\hline
  \end{tabular}
\caption{Computational results for small values of $n$. We report the size of the optimal antichain $\mathcal A$ that was found and its profile vector $(p_k)_{k\in K}$ where $p_k$ is the number of $k$-sets in $\mathcal A$. The last column contains the size of the antichain corresponding to the graphs constructed in Section \ref{subsec:l_4}. }\label{tab:small_cases}
\end{table}
We make the following observations.
\begin{enumerate}
\item For $n\leqslant 16$, the construction in Section \ref{subsec:l_4} yields optimal $\{2,3,4\}$-saturated graphs. So the bound in Conjecture \ref{conj:graph_conjecture} might be true for $\{2,3,4\}$-graphs even without the restriction to large $n$.
\item If $n\equiv 0\pmod 4$, i.e. when the construction in Section \ref{subsec:l_4} yields a regular graph, it is also optimal for $\{2,4\}$-saturation.
\item For $K=\{2,3,4\}$ all optimal examples that have been found do not contain any 3-set.
\item A particularly interesting case is $n=10$. The unique optimal $\{2,4\}$-saturated graph has the set of 4-cliques 
\[\mathcal A_4=\{1234,\,1567,\,2589,\,368a,\,479a\},\]
defining a $6$-regular graph in which every vertex is contained in exactly two 4-cliques, and every edge is contained in a unique 4-clique. This can not generalize to large values of $n$, since by the removal lemma in a $\{2,4\}$-saturated graph on $n$ vertices with the property that every edge is contained in at most a constant number of 4-cliques, the number of edges is $o(n^2)$. 
\end{enumerate}

\section{Asymptotic results}\label{sec:asymptotics}
The removal lemma argument that was used in \cite{gerbner2011saturating} to derive asymptotic lower bounds for the size of a maximal $\{l,l+1\}$-antichain in terms of Tur\'an numbers can be adapted to establish that for sets $K$ containing $\{2,3\}$ the construction in Section \ref{subsec:general_construction} gives the correct leading term.
\begin{theorem}\label{thm:strong_max_asymptotics}
Let $K$ be a fixed set containing 2 and 3, and let $l$ be its maximal element. For $K$-saturated graphs $G$ on $n$ vertices and $\mathcal A\in\mathfrak A(G)$, we have for $n\to\infty$,
\[\lvert E\rvert-\sum_{k\in K\setminus\{2\}}\lvert\mathcal A_k\rvert\leqslant\left(\frac{\lfloor l/2\rfloor\lceil l/2\rceil-1}{4\lfloor l/2\rfloor\lceil l/2\rceil}+o(1)\right)n^2=
\begin{cases}
\left(\frac{l^2-4}{4l^2}+o(1)\right)n^2  & \text{for even }l,\\
\left(\frac{l^2-5}{4l^2-4}+o(1)\right)n^2  & \text{for odd }l.
\end{cases}.\] 
\end{theorem}
\begin{proof}
Let $(G,\mathcal A)$ be a pair of a $K$-saturated graph $G=(V,E)$ and an antichain $\mathcal A\in\mathfrak A(G)$ which maximizes $\lvert E\rvert-\sum_{k\in K\setminus\{2\}}\lvert\mathcal A_k\rvert$. In particular, $\sum_{k\in K\setminus\{2\}}\lvert\mathcal A_k\rvert\leqslant n^2$, and since every triangle is contained in an element of some $\mathcal A_k$, and every $k$-clique contains $\tbinom{k}{3}$ triangles, the number of triangles in $G$ is bounded by 
\[\sum_{k\in K}\lvert\mathcal A_k\rvert\binom{k}{3}\leqslant n^2\sum_{k\in K\setminus\{2\}}\binom{k}{3}=o(n^3).\] 
By the removal lemma, we can make $G$ triangle-free by removing $o(n^2)$ edges, in other words there is a partition $E=E_0\cup E_1$ such that $\lvert E_1\rvert=o(n^2)$ and $E_0$ is triangle-free, and thus $\lvert E_0\rvert\leqslant n^2/4$ by Tur\'an's theorem. For $xy\in E_0$ and $k\in K$, let $\lambda_k(xy)$ denote the number of elements of $\mathcal A_k$ containing $xy$. Using the $K$-saturation of $G$ and observing that $E_0$, being triangle-free, contains at most $\lfloor k/2\rfloor\cdot\lceil k/2\rceil$ edges of any $k$-clique in $G$, we obtain
\[\lvert E_0\rvert\leqslant\sum_{xy\in E_0}\sum_{k\in K\setminus\{2\}}\lambda_k(xy)\leqslant\sum_{k\in K\setminus\{2\}}\lvert\mathcal A_k\rvert\lfloor k/2\rfloor\cdot\lceil k/2\rceil\leqslant \lfloor l/2\rfloor\cdot\lceil l/2\rceil\cdot\sum_{k\in K\setminus\{2\}}\lvert\mathcal A_k\rvert,\]
and consequently
\begin{multline*}
\lvert E\rvert-\sum_{k\in K\setminus\{2\}}\lvert\mathcal A_k\rvert=\lvert E_0\rvert+\lvert E_1\rvert-\sum_{k\in K\setminus\{2\}}\lvert\mathcal A_k\rvert\leqslant\left(1-\frac{1}{\lfloor l/2\rfloor\cdot\lceil l/2\rceil}\right)\lvert E_0\rvert+o(n^2)\\ \leqslant\frac{\lfloor l/2\rfloor\cdot\lceil l/2\rceil-1}{\lfloor l/2\rfloor\cdot\lceil l/2\rceil}\frac{n^2}{4}+o(n^2).\qedhere
\end{multline*}
\end{proof}
\begin{corollary}\label{cor:K_contains_3}
Let $K$ be a fixed set containing 2 and 3 with maximal element $l$, and let $\mathcal A$ be a maximal $K$-antichain of minimum size on $n$ points. Then
\[\lvert\mathcal A\rvert=\left(\frac{\lfloor l/2\rfloor\lceil l/2\rceil+1}{4\lfloor l/2\rfloor\lceil l/2\rceil}+o(1)\right)n^2.\]
\end{corollary}
\begin{proof}
  The lower bound follows from Theorem \ref{thm:strong_max_asymptotics} and the upper bound comes from the construction in Section \ref{subsec:general_construction}.
\end{proof}
\begin{corollary}\label{cor:strongly_maximal_asymptotics}
Let $K$ be a fixed set containing 2 with maximal element $l$, and let $\mathcal A$ be a strongly maximal $K$-antichain of minimum size on $n$ points. Then
\[\lvert\mathcal A\rvert=\left(\frac{\lfloor l/2\rfloor\lceil l/2\rceil+1}{4\lfloor l/2\rfloor\lceil l/2\rceil}+o(1)\right)n^2.\]
\end{corollary}
\begin{proof}
The lower bound follows from Corollary \ref{cor:K_contains_3}, because $\mathcal A$ is also a maximal $(K\cup\{3\})$-antichain, and the upper bound comes again from the construction in Section \ref{subsec:general_construction}.   
\end{proof}
We expect that the statement of Theorem \ref{thm:strong_max_asymptotics} remains true if the condition $3\in K$ is dropped, but in this case we can prove only a weaker bound. The main difficulty comes from the fact that we are not able to deduce that an optimal graph is essentially triangle-free (meaning the number of triangles is $o(n^3)$). This weakens both bounds that are essential in the proof of Theorem \ref{thm:strong_max_asymptotics}: the upper bound for the number of edges, and the lower bound for the number of cliques in terms of the edge number. The obvious starting point for deriving bounds is that an optimal graph is almost $K_k$-free for every $k\in K$, which by Tur\'an's theorem yields an asymptotic upper bound for the number of edges. Unfortunately, for $3\not\in K$ this bound is far away from the conjectured truth $n^2/4$. We now consider $K=\{2,4\}$ and derive a bound for $\lvert E\rvert-\lvert\mathcal C_4\rvert$ using lower bounds for the number of triangles in graphs with more than $n^2/4$ edges. Note that for sets $K$ of size $2$, the antichain $\mathcal A(G)$ is the only element of $\mathfrak A(G)$, so maximizing $\lvert E\rvert-\lvert\mathcal C_4\rvert$ is equivalent to (\ref{eq:reformulation}). For notational convenience, in the following we write $\gamma$ for $\lvert E\rvert/n^2$. Proceeding as in the proof of Theorem \ref{thm:strong_max_asymptotics}, we obtain a first bound.
\begin{lemma}\label{lem:first_bound}
Let $G$ be a $\{2,4\}$-saturated graph on $n$ vertices. Then
\begin{equation}\label{eq:first_bound}
  \lvert E\rvert-\lvert\mathcal C_4\rvert\leqslant\left(\frac{4\gamma}{5}+o(1)\right)n^2.
\end{equation}
\end{lemma}
\begin{proof}
As $\lvert\mathcal C_4\rvert=O(n^2)=o(n^4)$, $G$ can be made $K_4$-free by removing $o(n^2)$ edges, and we obtain a partition $E=E_0\cup E_1$ with $|E_1|=o(n^2)$ and $E_0$ $K_4$-free, in particular any 4-clique in $G$ contains at most 5 edges from $E_0$. Using the 4-saturation of $G$, we obtain
\[\lvert E_0\rvert\leqslant\sum_{xy\in E_0}\lambda_4(xy)\leqslant 5\lvert\mathcal C_4\rvert,\]
and finally $\lvert E\rvert-\lvert\mathcal C_4\rvert\leqslant (4/5)\lvert E_0\rvert+\lvert E_1\rvert=(4\gamma/5+o(1))n^2$.
\end{proof}
Note that this proof together with Tur\'an's theorem already implies $\lvert E\rvert-\lvert\mathcal C_4\rvert\leqslant(4/15+o(1))n^2$ for $\{2,4\}$-saturated graphs. In order to improve the inequality, we bound the number of triangles. First we have a lower bound in terms of the number of edges. In any graph $G$ with $n$ vertices and $\lvert E\rvert=\gamma n^2$ edges, the number of triangles is at least (see \cite{Fisher1989,Razborov2008}) 
\begin{equation*}
\frac{9\gamma-2-2(1-3\gamma)^{3/2}}{27}n^3.
\end{equation*}
On the other hand we get an upper bound for the number of triangles in terms of the number of 4-cliques.    
\begin{lemma}\label{lem:triples_and_quads}
Let $G$ be a $\{2,4\}$-saturated graph on $n$ vertices. Then the number of triangles is at most
\begin{equation*}
\frac{(n-4)\lvert\mathcal C_4\rvert}{3}+o(n^3).
\end{equation*}
\end{lemma}
\begin{proof}
Let $\mathcal B_3$ denote the set of triangles that are contained in a 4-clique in $G$, and let $\mathcal B'_3$ be the set of all other triangles. Note that $\lvert\mathcal B_3\rvert\leqslant 4\lvert\mathcal C_4\rvert=O(n^2)=o(n^3)$, so the number of triangles is dominated by $\lvert\mathcal B'_3\rvert$. Consider four vertices $a,b,c,d$ inducing a 4-clique in G, and let $x\in V\setminus\{a,b,c,d\}$ be another vertex. There is at most one triangle in $\mathcal B'_3$ of the form $xyz$ with $y,z\in\{a,b,c,d\}$. Consequently, there are at most $n-4$ triangles in $\mathcal B'_3$ having two vertices in $\{a,b,c,d\}$, and by summation over all 4-cliques the claim follows.
\end{proof}
Combining the two bounds for the number of triangles we obtain a second bound for $\lvert E\rvert-\lvert\mathcal C_4\rvert$ in terms of $\gamma$.
\begin{lemma}\label{lem:second_bound}
  Let $G$ be a $\{2,4\}$-saturated graph on $n$ vertices. Then
\begin{equation}\label{eq:second_bound}
\lvert E\rvert-\lvert\mathcal C_4\rvert\leqslant\left(\frac{2+2(1-3\gamma)^{3/2}}{9}+o(1)\right)n^2. 
\end{equation}
\end{lemma}
Now we can combine (\ref{eq:first_bound}) and (\ref{eq:second_bound}) to eliminate $\gamma$ and obtain an absolute bound for $\lvert E\rvert-\lvert\mathcal C_4\rvert$.  
\begin{theorem}\label{thm:weak_max_asymptotics}
Let $G$ be a $\{2,4\}$-saturated graph on $n$ vertices. Then
\[\lvert E\rvert-\lvert\mathcal C_4\rvert\leqslant\left(\frac{2(39+\sqrt{21})}{375}+o(1)\right)n^2<(0.232441+o(1))n^2.\]
\end{theorem}
\begin{proof}
The right hand sides of (\ref{eq:first_bound}) and (\ref{eq:second_bound}) are increasing and decreasing in $\gamma$, respectively. Equating them and solving for $\gamma$ yields $\gamma=(39+\sqrt{21})/150$ and the result follows. 
\end{proof}
\begin{corollary}
Let $\mathcal A$ be a maximal $\{2,4\}$-antichain of minimum size on $n$ points. Then
\[\left(\frac{219-4\sqrt{21}}{750}+o(1)\right)n^2\leqslant\lvert\mathcal A\rvert\leqslant\begin{cases}
\left\lceil (5n^2-16n)/16\right\rceil & \text{if $n$ is even,}\\
\left\lfloor(5n^2-14n)/16\right\rfloor & \text{if $n$ is odd.}\\
\end{cases}\]  
\end{corollary}
\begin{proof}
  The lower bound is a consequence of Theorem \ref{thm:weak_max_asymptotics}, and the upper bound comes from the construction in Section \ref{subsec:l_4}.
\end{proof}


\section{Open problems}\label{sec:problems}
An obvious next step is to improve the asymptotic bounds with the aim to replace the coefficient of $n^2$ in Theorem \ref{thm:weak_max_asymptotics} by $3/16$, which gives the leading term of the inequality in Conjecture \ref{conj:graph_conjecture}. It seems that arguments based on the removal lemma are not sufficient to reach this goal. One reason for this failure could be that the removal lemma only requires that there are only $o(n^4)$ 4-cliques, while in our situation we can use the stronger information that the numbers 4-cliques is $O(n^2)$. Another problem is to characterize the extremal graphs (resp. antichains) up to isomorphism. For the case $K=\{2,3\}$ this was done in \cite{GruettmuellerHartmannKalinowskiLeckRoberts2009}, and it would be nice to prove that for $K\in\left\{\{2,4\},\,\{2,3,4\}\right\}$ and large $n$ (some slight variations of) the constructions in Section \ref{subsec:l_4} yield all extremal graphs. In view of the computational results in Table \ref{tab:small_cases} one might also study the question if any maximal $K$-antichain of minimum size contains only $k$-sets and $l$-sets, where $k$ and $l$ are the smallest and the largest element of $K$, respectively.


\end{document}